\documentclass{article}
\usepackage[utf8]{inputenc}
\usepackage{enumerate}
\usepackage{amssymb, amsmath}
\usepackage{amsthm, thmtools, thm-restate}
\usepackage{hyperref, cleveref}
\usepackage{tikz}
\usetikzlibrary{arrows}
\usetikzlibrary{decorations.markings}
\usepackage[hmargin=3cm,vmargin=2.5cm]{geometry}
\newcommand{\tmcl}{\text{cl}_{\text{tm}}}
\newcommand{\dt}{D_{\mathcal{T}}}
\newcommand{\et}{E_{\mathcal{T}}}
\newcommand{\mt}{\mathcal{T}}

\newcommand{\fraisse}{Fra\"{i}ss\'{e} }
\newcommand*{\defeq}{\mathrel{\vcenter{\baselineskip0.5ex \lineskiplimit0pt
                     \hbox{\scriptsize.}\hbox{\scriptsize.}}}
                     =}

\declaretheorem[name=Theorem, refname={Theorem, Theorems}, numberwithin=section]{thm}
\declaretheorem[name=Lemma, refname={Lemma,Lemmas}, sibling=thm]{lem}
\declaretheorem[name=Fact, refname={Fact,Facts}, sibling=thm]{fact}

\declaretheorem[name=Corollary, sibling=thm]{cor}

\theoremstyle{definition}
\declaretheorem[name=Definition, refname={Definition,Definitions}, sibling=thm]{Def}

\title{$2^{\aleph_0}$ pairwise non-isomorphic maximal-closed subgroups of Sym$(\mathbb{N})$ via the classification of the reducts of the Henson digraphs}
\author{Lovkush Agarwal and Michael Kompatscher}
\date{\today}

\begin{document}
\maketitle
\begin{abstract}
Given two structures $\mathcal{M}$ and $\mathcal{N}$ on the same domain, we say that $\mathcal{N}$ is a reduct of $\mathcal{M}$ if all $\emptyset$-definable relations of $\mathcal{N}$ are $\emptyset$-definable in $\mathcal{M}$. In this article the reducts of the Henson digraphs are classified. Henson digraphs are homogeneous countable digraphs that omit some set of finite tournaments. As the Henson digraphs are $\aleph_0$-categorical, determining their reducts is equivalent to determining all closed supergroups $G<$ Sym$(\mathbb{N})$ of their automorphism groups. 

A consequence of the classification is that there are $2^{\aleph_0}$ pairwise non-isomorphic Henson digraphs 
which have no proper non-trivial reducts. Taking their automorphisms groups gives a positive answer to a question of Macpherson that asked if there are $2^{\aleph_0}$ pairwise non-conjugate maximal-closed subgroups of Sym$(\mathbb{N})$. By the reconstruction results of Rubin, these groups are also non-isomorphic as abstract groups.
\end{abstract}

This article contributes to the large body of work concerning the two intimately related topics of reducts of countable structures and of closed subgroups of Sym$(\mathbb{N})$.  Motivation for this work comes from both areas.

In the topic of reducts, the reducts of the Henson digraphs are classified up to first order interdefinability. This is the first time the reducts of uncountably many homogeneous structures have been classified. In all cases only finitely many reducts appear. This result supports a conjecture of Thomas in \cite{tho96} which says that all countable homogeneous structures in a finite relational language have only finitely many reducts. Evidence for this conjecture is building as there have been numerous classification results, e.g. \cite{cam76}, \cite{tho91}, \cite{tho96},  \cite{jz08}, \cite{bpp13}, \cite{ppp+}, \cite{aga14}. This conjecture is unresolved and continues to provide motivation for study.

In the topic of permutation groups, this article answers a question of Macpherson, Question 5.10 in \cite{bm15}, which asked to show that there are $2^{\aleph_0}$ pairwise non-conjugate maximal-closed subgroups of Sym$(\mathbb{N})$ with Sym($\mathbb{N}$) bearing the pointwise convergence topology. Several related questions have recently been tackled. Independently, \cite{bm15} and \cite{br13} showed that there exist non-oligomorphic maximal-closed subgroups of Sym$(\mathbb{N})$, the existence of which was asked in \cite{jz08}.  Also, independently, \cite{ks13} and \cite{br13} positively answered Macpherson's question of whether there are countable maximal-closed subgroups of Sym$(\mathbb{N})$.  One question that remains open is whether every proper closed subgroup of Sym($\mathbb{N}$) is contained in a maximal-closed subgroup of Sym($\mathbb{N}$), (Question 7.7 in \cite{mn96} and Question 5.9 in \cite{bm15}).

The main tool used in this classification of the reducts of the Henson digraphs is that of the so-called `canonical functions'. This Ramsey-theoretic tool was developed by Bodirsky and Pinsker to help analyse certain closed clones in relation to constraint satisfaction problems, a topic in theoretical computer science.  With further developments (\cite{bp11}, \cite{bpt13}), canonical functions have become powerful tools in studying reducts. The robustness and relative ease of the methodology is becoming more evident as several classifications have been achieved by their use, e.g. \cite{bpp13}, \cite{ppp+}, \cite{bb13}, \cite{aga14}, \cite{lp14}.

The description of $2^{\aleph_0}$ maximal-closed subgroups follows from the main theorem by taking the automorphism groups of a suitably modified version of Henson's (\cite{hen72}) construction of $2^{\aleph_0}$ pairwise non-isomorphic countable homogeneous digraphs. A short argument shows that their automorphism groups will be pairwise non-conjugate. However, we can say more: by Rubin's work on reconstruction (\cite{rub94}), the automorphism groups will be pairwise non-isomorphic as abstract groups. 

We outline the structure of the paper. In Section 1, we provide the necessary preliminary definitions and facts on the Henson digraphs, reducts and canonical functions. We also comment on some notational conventions that we use. In Section 2, we prove the main result of the article - the classification of the reducts of the Henson digraphs. In Section 2.1 we state the main result. In Section 2.2 we describe the reducts, establishing notation and important lemmas that are used in the rest of the paper. In Section 2.3 we carry out the combinatorial analysis of the possible behaviours of canonical functions. 2.4 contains the proof of the main theorem. In Section 3, we conclude by using the main theorem to show that there exist $2^{\aleph_0}$ maximal-closed subgroups of Sym$(\mathbb{N})$.

\section{Preliminaries}
\subsection{Notational Conventions}
If $A$ is a subset of $D$, $A^c$ denotes the complement of $A$ in $D$. We sometimes write `$ab$' as an abbreviation for $(a,b)$, e.g., we may write ``Let $ab$ be an edge of the digraph $D$''. Structures are denoted by $\mathcal{M},\mathcal{N}$, and their domains are $M$ and $N$ respectively. Sym$(M)$ is the set of all bijections $M \to M$ and Aut$(\mathcal{M})$ is the set of all automorphisms of $\mathcal{M}$. Given a formula $\phi(x,y)$, we use $\phi^*(x,y)$ to denote the formula $\phi(y,x)$. $S(\mathcal{M})$ denotes the space of types of the theory of $\mathcal{M}$. 
If $f$ has domain $A$ and $(a_1,\ldots,a_n) \in A^n$, then $f(a_1,\ldots,a_n) \defeq (f(a_1),\ldots,f(a_n))$. For $\bar{a},\bar{b} \in M^n$, we say $\bar{a}$ and $\bar{b}$ are isomorphic, and write $\bar{a} \cong \bar{b}$, to mean that the function $a_i \mapsto b_i$ for all $i$ such that $1 \leq i \leq n$ is an isomorphism.

There will be instances where we do not adhere to strictly correct notational usage, however, the meaning will be clear from the context. We highlight some examples. When using $n$-tuples, say $(a_1,\ldots,a_n) \in M^n$, we shall always assume that $a_i \neq a_j$ for all $i\neq j$.  We sometimes write `$a \in (a_1,\ldots,a_n)$' instead of `$a=a_i$ for some $i$ such that $1 \leq i \leq n$'. Another example is that we sometimes use $c$ to represent the singleton set $\{ c \}$ containing it. A fourth example is we may write `$\bar{a} \in A$' instead of `$\bar{a} \in A^n$ for some $n$'.

\subsection{Henson Digraphs}
A directed graph $(V,E)$, or digraph for short, is a set $V$ with an irreflexive anti-symmetric relation $E \subseteq V^2$. $V$ is the set of vertices, $E$ is the set of edges and we visualise an element $(a,b) \in E$ as being an edge going out of $a$ and into $b$. We say a digraph is empty if $E=\emptyset$. By $L_n$ we denote the linear order on $n$-elements, regarded as a digraph.

A tournament is a digraph in which there is an edge between every pair of distinct vertices. Throughout this article, $\mt$ will denote a set of finite tournaments. We will often refer to elements of $\mt$ as forbidden tournaments.

\begin{Def} \begin{enumerate}[(i)]
\item A structure $\mathcal{M}$ is \emph{homogeneous} if every isomorphism $f:A \to B$ between finite substructures $A,B$ of $\mathcal{M}$ can be extended to an automorphism $g \in$ Aut$(\mathcal{M})$.
\item For a structure $\mathcal{M}$, the \emph{age} of $\mathcal{M}$, Age($\mathcal{M})$, is the set of finite structures embeddable in $\mathcal{M}$.
\item Let $\mt$ be a set of tournaments. We let Forb$(\mt)$ be the set of finite digraphs $D$ such that for all $T \in \mt$, $D$ does not embed $T$.
\item We let $(\dt,\et)$ be the unique (up to isomorphism) countable homogeneous digraph whose age is Forb$(\mt)$. 
\item A \emph{Henson digraph} is a digraph isomorphic to $(\dt,\et)$ where $\mt$ is non-empty and does not contain the 1- or 2-element tournament.
\end{enumerate}
\end{Def}

The fact that $(\dt,\et)$ exists and is unique follows from the general \fraisse theory of amalgamation classes, developed by \fraisse in \cite{fra53}. This particular construction of digraphs was used by Henson in \cite{hen72} to show there exists uncountably many countable homogeneous digraphs. An accessible account on the theory of amalgamation classes can be found in \cite{hod97}. 

If $\mt=\emptyset$ then $(\dt,\et)$ is the generic digraph, the unique countable homogeneous digraph that embeds all finite digraphs. The reducts of the generic digraph are classified in \cite{aga14}. If $\mathcal{T}$ contains the 1-element tournament, then Forb($\mt)=\emptyset$. If $\mathcal{T}$ contains the 2-element tournament, then $(\dt,\et)$ is the countable empty digraph. These are degenerate cases which is why we defined the term Henson digraph to exclude these options.

\begin{lem} \label{hensondigraphs} Let $(D,E)$ be a Henson digraph.
\begin{enumerate}[(i)]
\item Th$(D,E)$ is $\aleph_0$-categorical.
\item $(D,E)$ is connected: for every distinct $a,b \in D$, there is a path from $a$ to $b$.
\end{enumerate}
\end{lem}

\begin{proof}
(i) The theory of any homogeneous structure in a finite relational language is $\aleph_0$-categorical. See \cite{hod97} for details.

(ii) Let $a,b \in D$ be distinct and without loss suppose that there is no edge between $a$ and $b$. Consider the finite digraph $\{a',b',c'\}$ such that there is no edge between $a'$ and $b'$, and there is an edge from $a'$ to $c'$ and from $c'$ to $b'$. Observe that $\{a',b',c'\}$ lies in Forb($\mt$), so is embeddable in $(D,E)$. By the homogeneity of $(D,E)$, we map $a'$ to $a$ and $b'$ to $b$ to obtain a $c \in D$ with $E(a,c)$ and $E(c,b)$.
\end{proof}

In order to use the canonical functions machinery, we need to expand the Henson digraphs to ordered digraphs. This is described in the following definition.

\begin{Def} \begin{enumerate}[(i)]
\item An \emph{ordered digraph} is a digraph which is also linearly ordered. Formally, it is a structure $(V,E,<)$ where $(V,E)$ is a digraph and $(V,<)$ is a linear order.
\item We let $(\dt, \et, <)$ be the unique (up to isomorphism) countable homogeneous ordered digraph such that a finite ordered digraph $(D,E,<)$ is embeddable in $(\dt,\et,<)$ iff $(D,E) \in$ Forb$(\mt)$.
\item We say $(D,E,<)$ is a \emph{Henson ordered digraph} if $(D,E,<) \cong (\dt,\et,<)$ for some $\mathcal{T}$.
\end{enumerate}
\end{Def}

\begin{fact} All Henson ordered digraphs are Ramsey structures.
\end{fact}

This fact follows by a direct application of the main theorem of \cite{nr83}. For the purposes of this article, it is not necessary to know what it means to be a Ramsey structure. The definition and examples of Ramsey structures can be found in \cite{jln14} and references therein. The importance of the Ramsey property and of introducing ordered digraphs will become evident in the Section 1.4.

\subsection{Reducts}
Let $\mathcal{M},\mathcal{N}$ be two structures on the same domain $M$. We say $\mathcal{N}$ is a \emph{reduct} of $\mathcal{M}$ if all $\emptyset$-definable relations in $\mathcal{N}$ are $\emptyset$-definable in $\mathcal{M}$. We say $\mathcal{N}$ is a proper reduct of $\mathcal{M}$ if $\mathcal{N}$ is a reduct of $\mathcal{M}$ but $\mathcal{M}$ is not a reduct of $\mathcal{N}$. In this article, if two structures $\mathcal{M}$ and $\mathcal{N}$ are both reducts of each other, we consider them to be the same structure. 

For any structure $\mathcal{M}$, the reducts of $\mathcal{M}$ form a lattice where $\mathcal{N} \leq \mathcal{N}'$ if $\mathcal{N}$ is a reduct of $\mathcal{N'}$. As well as classifying the reducts of  a Henson digraph, the lattice they form is also determined.

As a consequence of the theorem of Engeler, Ryll-Nardzewski and Svenonius (see \cite{hod97}), if $\mathcal{M}$ is $\aleph_0$-categorical then the lattice of reducts is anti-isomorphic to the lattice of closed groups $G$ such that Aut($\mathcal{M}) \leq G \leq$ Sym$(M)$.  This means that determining the lattice of reducts of the Henson digraphs is equivalent to determining the lattice of closed supergroups of the automorphism groups of the Henson digraphs.

We note that by closed we mean closed in the pointwise convergence topology on Sym$(M)$. Unravelling the definitions, this means that $F \subseteq$ Sym$(M)$ is closed if $F=$cl($F$), where $g\in$ Sym$(M)$ is in cl$(F)$ if for all finite $A \subset M$, there exists $f \in F$ such that $f(a)=g(a)$ for all $a \in A$.

\subsection{Canonical Functions} \label{canonical}
\begin{Def} Let $\mathcal{M}, \mathcal{N}$ be any structures. Let $f: M \to N$ be any function between the domains of the structures.
\begin{enumerate}[(i)]
\item The \emph{behaviour} of $f$ is the relation $\{(p,q) \in S(\mathcal{M})\times S(\mathcal{N}): \exists \bar{a} \in M, \bar{b} \in N$ such that tp$(\bar{a})=p$, tp$(\bar{b})=q$ and $f(\bar{a})=\bar{b} \}$.
\item If the behaviour of $f$ is a function $S(M) \to S(N)$, then we say $f$ is \emph{canonical}. Rephrased, we say $f$ is canonical if for all $\bar{a}, \bar{a}' \in M$, tp$(\bar{a})=$ tp$(\bar{a}') \Rightarrow$ tp$(f(\bar{a}))=$ tp$(f(\bar{a}'))$.
\item If $f$ is canonical, we use the same symbol $f$ to denote its behaviour.
\end{enumerate}
\end{Def}

For example, for any structure $\mathcal{M}$, every automorphism $f \in$ Aut$(\mathcal{M})$ is a canonical function, and for all types $p \in S(\mathcal{M})$, $f(p)=p$.

The benefit of canonical functions is that they are particularly well-behaved and can be easily manipulated and analysed. Furthermore, the next theorem, \autoref{blackbox}, essentially reduces the task of determining reducts to the task of analysing the behaviours of canonical functions.  In order to state the theorem, we need to give a couple of definitions.

\begin{Def}
Let $M$ be a countable set. Then $M^M$ is a topological monoid under the pointwise convergence topology and function composition. For $F \subseteq M^M$, we let $\tmcl(F)$, the topological monoid closure of $F$, denote the smallest closed monoid in $M^M$ containing $F$. If $M$ is the domain of a structure $\mathcal{M}$, we may abuse notation and write $\tmcl(F)$ for $\tmcl($Aut$(\mathcal{M}) \cup F)$.
\end{Def}

\begin{thm}\label{blackbox} Let $(D,E,<)$ be a Henson ordered digraph. Let $f \in$ Sym$(D)$ and $c_1,\ldots, c_n \in D$ be any vertices. Then there exists a function $g:D \to D$ such that \begin{enumerate}[(i)]
\item $g \in \tmcl($Aut$(D,E) \cup \{f\})$.
\item $g(c_i)=f(c_i)$ for $i=1,\ldots n$.
\item When regarded as a function from $(D,E,<,\bar{c})$ to $(D,E)$, $g$ is a canonical function.
\end{enumerate}
\end{thm}

This theorem is just an application of Lemma 14 in \cite{bpt13} to the Henson ordered digraphs. The toughest condition that needs to be checked in using Lemma 14 is that of being a Ramsey structure. As discussed earlier, the Henson ordered digraphs are indeed Ramsey structures.

\section{Classification of the Reducts}
For this section, we fix a Henson ordered digraph $(D,E,<)$ and let $\mt$ be its set of forbidden tournaments.

\subsection{Statement of Main Result}
\begin{Def}
\begin{enumerate}[(i)]
\item For $F \subseteq$ Sym($D$), let $\langle F \rangle$ denote the smallest closed subgroup of $G$ containing $F$. For brevity, when it is clear we are discussing supergroups of Aut$(D,E)$, we may abuse notation and write $\langle F \rangle$ to mean $\langle F$ $\cup $ Aut$(D,E) \rangle$.

\item  We let $\bar{E}(x,y)$ denote the underlying graph relation $E(x,y) \vee E(y,x)$. We let $N(x,y)$ denote the non-edge relation $\neg \bar{E}(x,y)$.

\item Assume $(D,E)$ is isomorphic to the digraph obtained by changing the direction of all its edges. In this case $- \in$ Sym$(D$) will denote a bijection such that for all $x,y \in D$, $E(-(x),-(y))$ iff $E(y,x)$.

\item Assume $(D,E)$ is isomorphic to the digraph obtained by changing the direction of all the edges adjacent to one particular vertex of $D$. In this case $sw \in$ Sym$(D$) will denote a bijection such that for some $a \in D$:
\[ E(sw(x),sw(y)) \text{ if and only if } \begin{cases}
  E(x,y) \text{ and } x,y \neq a, \text{ OR,}\\
  E(y,x) \text{ and } x=a \vee y=a
\end{cases}
\]
\end{enumerate}
\end{Def}

In words, $-$ is a function which changes the direction of all the edges of the digraph and $sw$ is a function which changes the direction of those edges adjacent to one particularly vertex.  The existence of $-$ or $sw$ depends on which tournaments are forbidden. This explains the wording of \autoref{main}(iii): if, for example, $-$ exists but $sw$ does not, then max$\{$Aut$(D,E), \langle - \rangle, \langle sw \rangle$, $\langle -,sw \rangle\}=\langle - \rangle$.


\begin{restatable}{thm}{maintheorem} \label{main}
Let $(D,E)$ be a Henson digraph and let $G \leq Sym(D)$ be a closed supergroup of Aut$(D,E)$. Then:
\begin{enumerate}[(i)]
\item $G \leq$ Aut$(D,\bar{E})$ or $G \geq$ Aut$(D,\bar{E})$
\item If $G <$ Aut$(D,\bar{E})$ then $G=$ Aut$(D,E), \langle - \rangle, \langle sw \rangle$ or $\langle -,sw \rangle.$
\item $(D,\bar{E})$ is the random graph, $(D,\bar{E})$ is a Henson graph or $(D,\bar{E})$ is not homogeneous. In the last case  Aut$(D,\bar{E})$ is equal to max$\{$Aut$(D,E), \langle - \rangle, \langle sw \rangle$, $\langle -,sw \rangle\}$ and is a maximal-closed subgroup of Sym($D$).
\end{enumerate}
\end{restatable}

The reducts of the random graph and the Henson graphs were classified by Thomas in \cite{tho91}. If $(D, \bar E)$ is the random graph, its only proper reducts are $\langle sw_{\Gamma} \rangle$ and $\langle -_{\Gamma} \rangle$, where $-_{\Gamma} \in$ Sym($D$) is a bijection which maps every edge to a non-edge and every non-edge to an edge and $sw_{\Gamma}$ is a bijection which does the same but only for those edges adjacent to a particular vertex $a \in D$. Henson graphs have no proper reducts. As an immediate consequence we get the following corollary of of \autoref{main}:

\begin{cor} Let $(D,E)$ be a Henson digraph. Then its lattice of reducts is a sublattice of the lattice below.  In particular, the lattice of reducts of $(D,E)$ is (isomorphic to) a sublattice of the lattice of reducts of the generic digraph (\cite{aga14}).

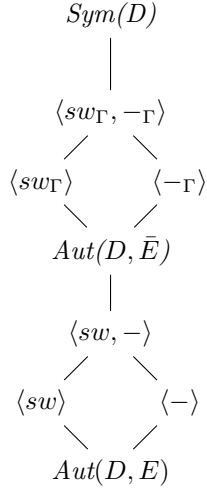
\begin{center}
\begin{tikzpicture}[node distance=1.3cm, auto]
  \node (D) {Aut$(D,E)$};
  \node (sw) [above left of=D] {$\langle sw \rangle$};
  \node (glo) [above right of=D] {$\langle - \rangle$};
  \node (sup) [above right of=sw] {$\langle sw,- \rangle$};
  \node (Ra) [above of=sup, yshift=-2mm] {Aut($D,\bar{E}$)};
  \node (sw2) [above left of=Ra] {$\langle sw_{\Gamma} \rangle$};
  \node (glo2) [above right of=Ra] {$\langle -_{\Gamma} \rangle$};
  \node (sup2) [above right of=sw2] {$\langle sw_{\Gamma},-_{\Gamma} \rangle$};
  \node (top) [above of=sup2] {Sym(D)};
  \draw[-] (D) to node {} (sw);
  \draw[-] (D) to node {} (glo);
  \draw[-] (sw) to node {} (sup);
  \draw[-] (glo) to node {} (sup);
  \draw[-] (sup) to node {} (Ra);
  \draw[-] (Ra) to node {} (sw2);
  \draw[-] (Ra) to node {} (glo2);
  \draw[-] (sw2) to node {} (sup2);
  \draw[-] (glo2) to node {} (sup2);
  \draw[-] (sup2) to node {} (top);
\end{tikzpicture}
\end{center}
\end{cor}

\subsection{Understanding the reducts}
In this section, we establish several important lemmas that play prominent roles in the proof of the main theorem. We omit the proofs of the lemmas for two reasons: They are relatively straightforward and are mostly identical to the lemmas in \cite[Section 3]{aga14}. Before we delve into the lemmas, we describe some terminology.

\begin{itemize}
\item Let $f,g: D \to D$ and $A \subseteq D$. We say \emph{$f$ behaves like $g$ on $A$} if for all finite tuples $\bar{a} \in A$, $f(\bar{a})$ is isomorphic (as a finite digraph) to $g(\bar{a})$. If $A=D$, we simply say \emph{$f$ behaves like $g$}. 
\item Let $A,B$ be disjoint subsets of $D$. We say $f$ behaves like $sw$ between $A$ and $B$ if $f$ switches the direction of all edges between $A$ and $B$ and preserves all non-edges between $A$ and $B$.
\item Let $A \subseteq D$. We let $sw_A: D \to D$ denote a function that behaves like $id$ on $A$ and $A^c$ and that behaves like $sw$ between $A$ and $A^c$. Note that the existence of $sw_A$ will depend on $A$ and on $\mt$.
\item We overload the symbols $-$ and $sw$ by letting them denote actions on finite tournaments. We say $\mt$ is closed under $-$ if for every $T \in \mt$, the tournament obtained from $T$ by changing the direction of all its edges is in $\mt$.  We say $\mt$ is closed under $sw$ if for every $T \in \mt$ and $t \in T$, the tournament obtained by changing the direction of those edges adjacent to $t$ is in $\mt$.
\end{itemize}

\begin{lem} \label{understanding-andsw}
\begin{enumerate}[(i)]
\item $-:D \to D$ exists if and only if $\mt$ is closed under $-$.
\item $sw:D \to D$ exists if and only if $\mt$ is closed under $sw$.
\item $\langle - \rangle \supseteq \{f \in$ Sym$(D): f$ behaves like $- \}$.
\item $\langle sw \rangle \supseteq \{f \in$ Sym$(D):$ there is $A \subseteq D$ such that $f$ behaves like $sw_A \}$.
\end{enumerate}
\end{lem}

\begin{lem} \label{equalsSymD} Let $G \leq$ Sym$(D)$ be a closed supergroup of Aut$(D,E)$.
\begin{enumerate}[(i)]
\item If $G$ is $n$-transitive for all $n \in \mathbb{N}$, then $G=$ Sym$(D)$. Note that $G$ is $n$-transitive if for all pairs of tuples $\bar{a}, \bar{b} \in D^n$, there exists $g \in G$ such that $g(\bar{a})=\bar{b}$.
\item If $G$ is $n$-homogeneous for all $n \in \mathbb{N}$, then $G=$ Sym$(D)$. Note that $G$ is $n$-homogeneous if for all subsets $A,B \subset D$ of size $n$, there exists $g \in G$ such that $g(A)=B$.
\item Suppose that whenever $A \subset D$ is finite and has edges, there exists $g \in G$ such that $g(A)$ has less edges than in $A$. Then $G=$ Sym$(D)$.
\item Suppose that there exists a finite $A \subset D$ and $g \in G$ such that $g$ behaves like $id$ on $D\backslash A$, $g$ behaves like $id$ between $A$ and $D\backslash A$, and, $g$ deletes at least one edge in $A$. Then, $G=$ Sym$(D)$.
\end{enumerate}
\end{lem}

\textbf{Terminology}. Let $a_1,\ldots,a_n, b_1, \ldots, b_n \in D$. We say $\bar{a}$ and $\bar{b}$ are isomorphic as graphs if $\bar{E}(a_i,a_j) \leftrightarrow \bar{E}(b_i,b_j)$ for all $i,j$.

\begin{lem} \label{understandingN} Let $G \leq$ Sym$(D)$ be a closed supergroup of Aut$(D,E)$.
\begin{enumerate}[(i)]
\item Suppose that whenever $\bar{a}$ and $\bar{b}$ are isomorphic as graphs, there exists  $g \in G$ such that $g(\bar{a})=\bar{b}$. Then $G \geq$ Aut$(D,\bar{E})$.
\item Suppose that for all $A=\{a_1,\ldots,a_n\} \subset D$, there exists $g \in G$ such that for all edges $a_ia_j$ in $A$, $E(g(a_i),g(a_j))$ iff $i<j$. (Intuitively, such a $g$ is aligning the edges so they all point in the same direction.) Then, $G \geq$ Aut$(D,\bar{E})$.
\item Suppose that for all finite $A \subset D$ and all edges $aa' \in A$ there is $g \in G$ such that $g$ changes the direction of $aa'$ and behaves like $id$ on all other edges and non-edges of $A$. Then $G \geq$ Aut$(D,\bar{E})$.
\item Suppose there is a finite $A \subset D$ and a $g \in G$ such that $g$ behaves like $id$ on $D \backslash A$, $g$ behaves like $id$ between $A$ and $D \backslash A$, and $g$ switches the direction of some edge in $A$. Then, $G \geq$ Aut$(D,\bar{E})$.
\end{enumerate}
Furthermore, in all of these cases we can also conclude that the underlying graph $(D,\bar{E})$ is homogeneous.
\end{lem}

\subsection{Analysis of Canonical Functions}
To help motivate the analysis we are about to undertake, we sketch a part of the proof of the main theorem. One task will be to show that if $G>$Aut$(D,E)$ then $G>\langle - \rangle$ or $G>\langle sw \rangle$. Since $G>$Aut$(D,E)$, $G$ does not preserve the relation $E$, so there exist $g \in G$ and $c_1,c_2 \in D$ witnessing this. Then by \autoref{blackbox}, we find a canonical function $f: (D,E,<,c_1,c_2) \to (D,E)$ that agrees with $g$ on $(c_1,c_2)$ and which is generated by $G$. The behaviour of $f$ will give us information about $G$. We only have to consider the behaviour of $f$ on the 2-types, since $(D,E,<,c_1,c_2)$ has quantifier elimination and all relations are of arity $\leq 2$. Therefore $f$ has finitely many possible behaviours, so we can check each case and show that $G$ must contain $\langle - \rangle$ or $\langle sw \rangle$.

\subsubsection{Canonical functions from $(D,E,<)$}
We start our analysis of the behaviours with the simplest case, which is when no constants are added.

\textbf{Notation and facts.} \begin{itemize}
\item Let $\phi_1(x,y),\ldots,\phi_n(x,y)$ be formulas. We let $p_{\phi_1,\ldots,\phi_n}(x,y)$ denote the (partial) type determined by the formula $\phi_1(x,y) \wedge \ldots \wedge \phi_n(x,y)$.
\item There are three 2-types in $(D,E)$: $p_{E},p_{E^*}$ and $p_{N}$.
\item There are six 2-types in $(D,E,<)$: $p_{<,E},p_{<,E^*}, p_{<,N}$, $p_{>,E},p_{>,E^*}$ and $p_{>,N}$.
\end{itemize}

\begin{lem} \label{noconstants} Let $G$ be a closed supergroup of Aut$(D,E)$, let $f \in \tmcl(G)$, and let $f$ be canonical when considered as a function from $(D,E,<)$ to $(D,E)$.
\begin{enumerate}[(i)]
\item If $f(p_{<,N})=p_N, f(p_{<,E})=p_{E^*}$ and $f(p_{<,E^*})=p_{E}$, then $- \in G$. In particular, $-$ exists.
\item If $f(p_{<,N})=p_N, f(p_{<,E})=p_{E}$ and $f(p_{<,E^*})=p_{E}$, then $(D,\bar{E})$ is a homogeneous graph and $G \geq$ Aut$(D,\bar{E})$.
\item If $f(p_{<,N})=p_N, f(p_{<,E})=p_{E^*}$ and $f(p_{<,E^*})=p_{E^*}$, then $(D,\bar{E})$ is a homogeneous graph and $G \geq$ Aut$(D,\bar{E})$.
\item If $f(p_{<,N})=p_E$ or $p_{E^*}, f(p_{<,E})=p_{N}$ and $f(p_{<,E^*})=p_{N}$, then $(D,\bar{E})$ is a homogeneous graph and $G \geq$ Aut$(D,\bar{E})$.
\item If $f$ has any other non-identity behaviour, then either we get a contradiction (i.e. that behaviour is not possible) or $G=$Sym($D$).
\end{enumerate}
\end{lem}



\begin{proof}
(i) It follows from \autoref{understanding-andsw} that $-$ does indeed exist. That $- \in G$ follows straightforwardly from the definitions.

(ii), (iii) These follow straightforwardly using \autoref{understandingN}.

(iv) By considering $f^2$ this case reduces to either (ii) or (iii).

(v) Case 1: $f(p_{<,N})=p_N$.  We are left with the behaviours where $f(p_{<,E})=p_{N}$ or $f(p_{<,E^*})=p_{N}$ (or both), as all the other possibilities have been dealt with above.   Now for any finite $A \subset D$ that has edges, $f(A)$ has less edges than $A$ does. So by \autoref{equalsSymD}, we conclude that $G=$ Sym$(D)$.

Case 2: $f(p_{<,N})=p_{E}$\\
Case 2a: $f(p_{<,E})=p_{E}$ and $f(p_{<,E^*})=p_{E}$. 
For every $\bar{a},\bar{b} \in D^n$, $f(\bar{a}) \cong f(\bar{b}) \cong L_n$ (the $n$-element linear order), so $G$ is $n$-transitive for all $n$, so $G=$ Sym$(D)$.

Case 2b: $f(p_{<,E})=p_{E^*}$ and $f(p_{<,E^*})=p_{E^*}$. Consider $f^2$ and use the same argument as in Case 2a to show that $G=$ Sym$(D)$.

Case 2c: $f(p_{<,E})=p_{E}$ and $f(p_{<,E^*})=p_{E^*}$. We will show that this behaviour is not possible. Let $T \in \mathcal{T}$ be of minimal cardinality. Enumerate $T$ as $T=(t_1,\ldots,t_n)$ so that we have an edge going from $t_1$ to $t_2$ (as opposed to $t_2$ to $t_1$).  Now let $A=(a_1,\ldots,a_n)$ be the ordered digraph constructed as follows: Start with $T$, delete the edge $t_1t_2$, and add a linear order so that $a_1 < a_2$.  As $T$ was minimal, $A$ can be embedded in $(D,E,<)$, so then $f(A) \subset (D,E)$. But by the construction of $A$, $f(A) \cong T$, so we have shown that $T$ is embeddable in $(D,E)$. This contradicts that $T \in \mathcal{T}$.

Case 2d: $f(p_{<,E})=p_{E^*}$ and $f(p_{<,E^*})=p_{E}$. Considering $f^2$ reduces to a case that is dual to Case 2c.

Case 2e: $f(p_{<,E})=p_{E}$ and $f(p_{<,E^*})=p_{N}$. Considering $f^2$ reduces to Case 2a.

Case 2f: $f(p_{<,E})=p_{N}$ and $f(p_{<,E^*})=p_{E}$. Considering $f^2$ reduces to Case 1. 

Case 2g: $f(p_{<,E})=p_{E^*}$ and $f(p_{<,E^*})=p_{N}$. We will show that this behaviour is not possible. Let $T \in \mathcal{T}$ be of minimal cardinality. Observe that $f^3$ has the identity behaviour, so that $f^3(T) = T$.  Now observe that $f^2(T)$ is a digraph that contains non-edges, so by the minimality of $T$, $f^2(T)$ can be embedded in $(D,E,<)$. But then applying $f$ shows that $f(f^2(T))$ is embeddable in $(D,E)$, i.e. that $f^3(T) = T$ is embeddable in $(D,E)$. This contradicts that $T \in \mt$.

Case 2h: $f(p_{<,E})=p_{N}$ and $f(p_{<,E^*})=p_{E^*}$. Using the same argument as in 2g shows that this case is not possible.

Case 3: $f(p_{<,N})=p_{E^*}$. This case is is symmetric to Case 2.
\end{proof}

\subsubsection{Canonical functions from $(D,E,<,\bar{c})$}
We now move on to the general situation where we have added constants $\bar{c} \in D$ to the structure. For convenience, we assume that $c_i<c_j$ for all $i<j$. Since $(D,E)$ is $\aleph_0$-categorical, $(D,E,\bar{c})$ is also $\aleph_0$-categorical, so the $n$-types of $(D,E,<,\bar{c})$ correspond to the orbits of Aut$(D,E,<,\bar{c})$ acting on the set of $n$-tuples of $D$.  For this reason, we often conflate the notion of types and orbits.

We need to describe the 2-types of $(D,E,<,\bar{c})$, and to do that we first need to describe the 1-types. There are two kinds of 1-types, i.e. two kinds of orbits. The first is a singleton, e.g. $\{c_1\}$. The other orbits are infinite and are determined by how their elements are related to the $c_i$, e.g., one of the infinite orbits could be $\{x \in D: x<c_1 \wedge \bigwedge_i E(x,c_i)\}$. 

Unlike in the case of the generic digraph, these orbits will not necessarily be isomorphic to the original structure. For example, let $\mathcal{T}=\{L_3\}$ and $\bar{c}=(c_1)$.  Then consider the orbit $X=\{x \in D: x <c_1 \wedge E(x,c_1)\}$.  If there was an edge, $ab$ say, in $X$, then $\{c_1,a,b\}$ would be a copy of $L_3$. However, $L_3$ is forbidden. Thus, $X$ contains no edges so in particular $X$ is not isomorphic to $(\dt,\et,<)$.

However, there are some orbits that are isomorphic to the original structure. For example, regardless of $\mt$, the orbit $X=\{x \in D: x<c_1 \wedge \bigwedge_i N(x,c_i)\}$ is isomorphic to $(D,E,<)$. These orbits form a central part of the argument so we give them a definition.

\begin{Def} Let $\bar{c} \in D$ and $X \subset D$ be an orbit of $(D,E,<,\bar{c})$. We say $X$ is \emph{independent} if $X$ is infinite and there are no edges between $\bar{c}$ and $X$.
\end{Def}

The following lemma highlights the key feature of independent orbits that makes them useful.

\begin{lem} \label{indorb} Let $X$ be an independent orbit of $(D,E,<,\bar{c})$ and let $v \in D \backslash (X \cup \bar{c})$. Let $A=(a_0,\ldots,a_n)$ be a finite digraph embeddable in $D$. Then there are $x_1,\ldots,x_n \in X$ such that $(a_0,a_1,\ldots,a_n) \cong (v,x_1,\ldots,x_n)$.
\end{lem}

\begin{proof} Let $k$ be the length of the tuple $\bar{c}$. Consider the finite digraph $A'$ which is constructed as follows: start with $A$, add new vertices $c_1',\ldots,c_k'$ and then add edges so that we have $(a_0,c_1',\ldots,c_k') \cong (v,c_1,\ldots,c_k)$ and so that there are no edges between $a_j$ and $c_i$ for all $i$ and all $j>0$. 

Any tournament embeddable in $A'$ must be embeddable in either $\{a_0,a_1,\ldots,a_n\}$ or $\{a_0,c_1',\ldots,c_k'\}$, which are both in the age of $(D,E)$. Therefore $A'$ is also in the age of $(D,E)$. Adding a suitable linear order we get a ordered digraph that is isomorphic to the desired  $(v,x_1,\ldots,x_n) \cup \bar c$. By the homogeneity of $(D,E,<)$ we are done.
\end{proof}

\textbf{Notation} Let $A,B$ be definable subsets of $D$ and let $\phi_1(x,y),\ldots,\phi_n(x,y)$ be formulas. We let $p_{A,B,\phi_1,\ldots,\phi_n}(x,y)$ denote the (partial) type determined by the formula $x \in A \wedge y \in B \wedge \phi_1(x,y) \wedge \ldots \wedge \phi_n(x,y)$.

Using this notation, we can describe the 2-types of $(D,E,<,\bar{c})$. They are all of the form $p_{X,Y,\phi, \psi}= \{(a,b) \in D: a \in X, b \in Y, \phi(a,b)$ and $\psi(a,b)\}$, where $X$ and $Y$ are orbits, $\phi \in \{<,>\}$ and $\psi \in \{E,E^*,N\}$.

Our task now is to analyse the possibilities for $f(p_{X,Y,\phi,\psi})$, where $f$ is a canonical function. It turns out that it is sufficient to study those cases where we assume $X$ is an independent orbit. The first lemma deals with the situation when $X=Y$.

\begin{lem} \label{oneorbit} Let $G$ be a closed supergroup of Aut$(D,E)$, let $\bar{c} \in D$, let $f \in \tmcl(G)$, and let $f$ be canonical when considered as a function from $(D,E,<,\bar{c})$ to $(D,E)$.  Let $X \subset D$ be an independent orbit.
\begin{enumerate}[(i)]
\item If $f(p_{X,X,<,N})=p_N, f(p_{X,X,<,E})=p_{E^*}$ and $f(p_{X,X,<,E^*})=p_{E}$, then $- \in G$. In particular, $-$ exists.
\item If $f(p_{X,X,<,N})=p_N, f(p_{X,X,<,E})=p_{E}$ and $f(p_{X,X,<,E^*})=p_{E}$, then $(D,\bar{E})$ is a homogeneous graph and $G \geq$ Aut$(D,\bar{E})$.
\item If $f(p_{X,X,<,N})=p_N, f(p_{X,X,<,E})=p_{E^*}$ and $f(p_{X,X,<,E^*})=p_{E^*}$, then $(D,\bar{E})$ is a homogeneous graph and $G \geq$ Aut$(D,\bar{E})$.
\item If $f(p_{X,X,<,N})=p_E$ or $p_{E^*}, f(p_{X,X,<,E})=p_{N}$ and $f(p_{X,X,<,E^*})=p_{N}$, then $(D,\bar{E})$ is a homogeneous graph and $G \geq$ Aut$(D,\bar{E})$.
\item If $f$ has any other non-identity behaviour, then either we get a contradiction or $G=$Sym($D$).
\end{enumerate}
\end{lem}

\begin{proof} The proof is identical to that of \autoref{noconstants} because $X \cong (D,E,<)$.
\end{proof}

Next we look at the behaviour of $f$ between an independent orbit $X$ and any other orbit $Y$. This task is split depending on how $X$ and $Y$ relate with regard to the linear order.

\textbf{Facts and Notation} There are two ways that two infinite orbits $X$ and $Y$ of Aut($D,E,<,\bar{c})$ can relate to each other with respect to the linear order $<$:
\begin{itemize}
\item All of the elements of one orbit, $X$ say, are smaller than all of the elements of $Y$. This is abbreviated by `$X<Y$'.
\item $X$ and $Y$ are interdense: $\forall x<x' \in X, \exists y \in Y$ such that $x<y<x'$ and vice versa.
\end{itemize}

The next lemma contains the analysis for the case where $X<Y$ or $X>Y$.

\begin{lem} \label{xlessthany}Let $G$ be a closed supergroup of Aut$(D,E)$, let $\bar{c} \in D$, let $f \in \tmcl(G)$, and let $f$ be canonical when considered as a function from $(D,E,<,\bar{c})$ to $(D,E)$.  Let $X \subset D$ be an independent orbit on which $f$ behaves like $id$ and let $Y$ be an infinite orbit such that $X<Y$ or $X>Y$.
\begin{enumerate}[(i)]
\item If $f(p_{X,Y,N})=p_N, f(p_{X,Y,E})=p_{E^*}$ and $f(p_{X,Y,E^*})=p_{E}$, then $sw \in G$. In particular, $sw$ exists.
\item If $f(p_{X,Y,N})=p_N, f(p_{X,Y,E})=p_{E}$ and $f(p_{X,Y,E^*})=p_{E}$, then $(D,\bar{E})$ is a homogeneous graph and $G \geq$ Aut$(D,\bar{E})$.
\item If $f(p_{X,Y,N})=p_N, f(p_{X,Y,E})=p_{E^*}$ and $f(p_{X,Y,E^*})=p_{E^*}$, then $(D,\bar{E})$ is a homogeneous graph and $G \geq$ Aut$(D,\bar{E})$.
\item If $f(p_{X,Y,N})=p_E$ or $p_{E^*}, f(p_{X,Y,E})=p_{N}$ and $f(p_{X,Y,E^*})=p_{N}$, then $(D,\bar{E})$ is a homogeneous graph and $G \geq$ Aut$(D,\bar{E})$.
\item If $f$ has any other non-identity behaviour, then either we get a contradiction or $G=$Sym($D$).
\end{enumerate}
\end{lem}
Remark: We do not need to include $<$ or $>$ in the subscripts of the type because it is automatically determined by how $X$ and $Y$ are related to $\bar{c}$.

\begin{proof} Assume that $X<Y$. The proof for the case $Y<X$ is symmetric. Let $y_0 \in Y$ be any element.

(i) The proof is analogous to that of Case (i) in \autoref{noconstants} and is left as an exercise for the reader. Note that \autoref{indorb} is needed for this.

(ii) Using \autoref{understandingN} (ii), it suffices to show that for any finite $A \subset D$ we can align all its edges by using functions in $G$. Let $A=\{a_1,\ldots,a_n\}$. First we map $a_{n-1}$ to $y_0$ and the rest of $A$ into $X$ (possible by \autoref{indorb}), and then apply $f$.  Then we repeat but with $a_{n-2}$ instead of $a_{n-1}$, then with $a_{n-3}$, and so on until $a_1$.

(iii) Same as (ii).

(iv) The same argument as in (ii) works but with a slight modification: the intuition is that whenever $f$ was applied to some tuple $(a_0,\ldots,a_n)$ in those proofs, here we apply $f$ twice to get the same effect. To be more precise, the modification is as follows.  Let $(a_0,\ldots,a_n) \in D$. We first map this to an isomorphic copy $(y_0,x_1,\ldots,x_n)$ for some $x_i \in X$. Then apply $f$. Then again we map this to an isomorphic tuple $(y_0,x_1',\ldots,x_n')$ for some $x_i' \in X$. Then apply $f$ a second time. The total effect of this procedure is the same as what the canonical function did in Case (ii) or (iii). Thus we have reduced this case to either (ii) or (iii).

Remark: For the rest of this proof, we will use the phrase ``by applying $f$ twice'' to refer to the procedure described above.

(v) Case 1: $f(p_{<,N})=p_N$.  By a similar argument as in Case 1 of \autoref{noconstants}, $G$=Sym$(D)$. Note that \autoref{indorb} is needed for this.

Case 2: $f(p_{X,Y,N})=p_{E}$\\
Case 2a: $f(p_{X,Y,E})=p_{E^*}$.  We will show that this behaviour is not possible, in a similar fashion to Case 2c of \autoref{noconstants}. Let $T \in \mt$ be of minimal size and enumerate $T$ as $(t_0,t_1,\ldots,t_n)$ so that $t_0$ has at least one edge going into it. Construct a digraph $A=(a_0,a_1,\ldots,a_n)$ as follows: start with $A$ being equal to $T$ and then replace edges into $a_0$ with non-edges, replace edges out of $a_0$ with incoming edges, and leave all other edges of $A$ the same.

Since $T$ was minimal, $A \in$ Forb$(\mt)$ so $A$ can be embedded in $D$. Furthermore, by \autoref{indorb} there are $x_i \in X$ such that $(a_0,a_1,\ldots,a_n) \cong (y_0,x_1,\ldots,x_n)$.  Now apply $f$. By construction of $A$, $f(y_0,x_1,\ldots,x_n) \cong (t_0,\ldots,t_n)$.  Thus, $T$ is embeddable in $D$, contradicting $T \in \mt$.

Case 2b: $f(p_{X,Y,E^*})=p_{E^*}$. Use the same argument as Case 2a to show this is not possible.

Now there are only three behaviours left to analyse.

Case 2c: $f(p_{X,Y,E})=p_{E}$ and $f(p_{X,Y,E^*})=p_{E}$. We will show that $G=$ Sym$(D)$, by showing that every tuple $(a_0,\ldots,a_{n-1}) \in D^n$ can be mapped to $L_n$ using functions in $G$. We do this by induction on $n$.  The base case $n=1$ is trivial so let $n>1$. By the inductive hypothesis we can assume that $(a_1,\ldots,a_{n-1}) \cong L_{n-1}$. By \autoref{indorb} we map $\bar{a}$ to an isomorphic tuple $(y_0,x_1,\ldots,x_{n-1})$ for some $x_i \in X$. Then applying $f$ maps the tuple to a copy of $L_n$, as required. 

Case 2d: $f(p_{X,Y,E})=p_{E}$ and $f(p_{X,Y,E^*})=p_{N}$. By applying $f$ twice this case is reduced to Case 2c.

Case 2e: $f(p_{X,Y,E})=p_{N}$ and $f(p_{X,Y,E^*})=p_{E}$. By applying $f$ twice this case is reduced to Case 1.

Case 3: $f(p_{X,Y,N})=p_{E^*}$. This case is symmetric to Case 2.
\end{proof}

In the proof above we only had to study the behaviour of $f$ on $\{y_0\} \cup X$ for some fixed $y_0 \in Y$. This was basically a consequence of \autoref{indorb}. We remark that we will use the arguments in this proof with minimal modification to prove subsequent lemmas, including the following, where we have two interdense orbits.

\begin{lem} \label{xyinterdense}
Let $G$ be a closed supergroup of Aut$(D,E)$, let $\bar{c} \in D$, let $f \in \tmcl(G)$, and let $f$ be canonical when considered as a function from $(D,E,<,\bar{c})$ to $(D,E)$.  Let $X \subset D$ be an independent orbit on which $f$ behaves like $id$ and let $Y$ be an infinite orbit such that $X$ and $Y$ are interdense. Then at least one of the following holds.
\begin{enumerate}[(i)]
\item $f$ preserves all the edges and non-edges between $X$ and $Y$
\item $f$ switches the direction of all the edges between $X$ and $Y$. In particular, $sw$ exists.
\item $G\geq$Aut$(D,\bar{E})$ and $(D,\bar{E})$ is a homogeneous graph.
\item $G=$ Sym$(D)$.
\end{enumerate}
\end{lem}

\begin{proof}
First just consider the increasing tuples from $X$ to $Y$. With the same arguments as in \autoref{xlessthany} one can show that either
\begin{enumerate}[(a)]
\item $f(p_{X,Y,N,<})=p_N, f(p_{X,Y,E,<})=p_{E}$ and $f(p_{X,Y,E^*,<})=p_{E*}$,
\item $f(p_{X,Y,N,<})=p_N, f(p_{X,Y,E,<})=p_{E^*}$ and $f(p_{X,Y,E^*,<})=p_{E}$,
\item $G\geq$Aut$(D,\bar{E})$ and $(D,\bar{E})$ is a homogeneous graph, or
\item $G=$ Sym$(D)$.
\end{enumerate}
If (c) or (d) is true we are done, so assume (a) or (b) is true.  Similarly we can assume that $f$ behaves like $id$ or $sw$ between decreasing tuples from $X$ to $Y$. Thus it remains to check, without loss, what happens if $f$ behaves like $id$ on decreasing tuples and $sw$ on increasing tuples. Explicitly we are asssuming that:

$f(p_{X,Y,N,<})=p_N, f(p_{X,Y,E,<})=p_{E^*}$, $f(p_{X,Y,E^*,<})=p_{E}$. and \\
$f(p_{X,Y,N,>})=p_N, f(p_{X,Y,E,>})=p_{E}$,  $f(p_{X,Y,E^*,>})=p_{E^*}$.

Let $\bar a = (a_0,a_1,\ldots,a_n) \in$ Forb($\mt$) be a digraph with at least one edge $E(a_0,a_1)$. We can consider $\bar a$ as an ordered digraph by setting $a_i < a_j \leftrightarrow i < j$. Then by \autoref{indorb} $\bar a$ has an isomorphic copy $\bar b = (b_0,b_1,\ldots, b_n)$ such that $b_1 \in Y$ and $b_i \in X$ for $i \neq 1$. All the edges of $\bar b$ are preserved under $f$, except for the edge $E(b_0,b_1)$ whose direction is switched. By \autoref{understandingN}, we conclude that $G\geq$Aut$(D,\bar{E})$ and $(D,\bar{E})$ is a homogeneous graph.
\end{proof}

We end by looking at how $f$ can behave between the constants $\bar{c}$ and the rest of the structure.

\begin{lem} \label{constants} Let $G$ be a closed supergroup of Aut$(D,E)$, let $(c_1,\ldots,c_n) \in D$, let $f \in \tmcl(G)$, and let $f$ be canonical when considered as a function from $(D,E,<,\bar{c})$ to $(D,E)$.  Suppose that $f$ behaves like $id$ on $D^- \defeq D\backslash \{c_1,\ldots,c_n\}$. Then at least one of the following holds.
\begin{enumerate}[(i)]
\item For all $i, 1\leq i \leq n$, $f$ behaves like $id$ or like $sw$ between $c_i$ and $D^-$.
\item $G\geq$ Aut$(D,\bar{E})$ and $(D,\bar{E})$ is a homogeneous graph.
\item $G=$ Sym$(D)$.
\end{enumerate}
\end{lem}

\begin{proof} Fix some $i$, $1 \leq i \leq n$. Let $X_{out} = \{x \in D: x<c_1 \wedge E(c_i,x) \wedge \bigwedge_{j \neq i} N(c_j,x)\}$. Define $X_{in}$ and $X_{N}$ similarly, with $E(c_i,x)$ replaced with $E(x,c_i)$ and $N(x,c_i)$ respectively. Then for any finite digraph $(a_0,a_1,\ldots,a_n)$, there exist $x_1,\ldots,x_n \in X_{out} \cup X_{in} \cup X_{N}$ such that $(a_0,a_1,\ldots,a_n) \cong (c_i,x_1,\ldots,x_n)$. So by replicating the proof of \autoref{xlessthany} we can assume that $f$ behaves like $id$ or $sw$ between $c_i$ and $X_{out} \cup X_{in} \cup X_{N}$.  Without loss, we assume $f$ behaves like $id$ - the argument for $sw$ will be the same.

If $f$ behaves like $id$ between $c_i$ and $D^-$ we are done, so suppose there is an infinite orbit $X$ such that $f$ does not behave like $id$ between $c_i$ and $X$. Assume that there are edges from $c_i$ into $X$ - the arguments for the other two cases are similar.

Let $A$ be any finite digraph in the age of $D$ and let $ab$ be any edge in $A$.  Then observe that there is an embedding of $A$ into $D$ such that $a$ is mapped to $c_i$, $b$ is mapped into $X$, and the rest of $A$ is mapped into $X_{out} \cup X_{in} \cup X_{N}$.  Then applying $f$ changes exactly the one edge $ab$ in $A$, so by \autoref{equalsSymD} or \autoref{understandingN} as appropriate, we are done.
\end{proof}

\subsection{Proof of main result}
\maintheorem*

\begin{proof} (i) Suppose for contradiction that $G \not\geq$ Aut$(D,\bar{E})$ and $G \not\leq$ Aut$(D,\bar{E})$. Because of the second assumption $G$ violates the relation $\bar{E}$. By \autoref{blackbox} this can be witnessed by a canonical function. Precisely, this means there are $c_1,c_2 \in D$ and $f \in \tmcl(G)$ such that $f: (D,E,<,c_1,c_2) \to (D,E)$ is a canonical function, $\bar{E}(c_1,c_2)$ and $N(f(c_1),f(c_2))$.

Now let $X$ be an independent orbit of $(D,E,<,c_1,c_2)$.

\textbf{Claim 1.} $f$ behaves like $id$ on $X$.

By \autoref{oneorbit} we know that $f$ behaves like $id$ or $-$ on $X$, otherwise $G$ would contain Aut$(D,\bar{E})$. If $f$ behaves like $-$ on $X$, then we continue by replacing $f$ by $- \circ f$.

\textbf{Claim 2.} $f$ behaves like $id$ between $X$ and every other infinite orbit $Y$.

Let $Y$ be another infinite orbit. By \autoref{xlessthany} and \autoref{xyinterdense}, $f$ behaves like $id$ or $sw$ between $X$ and $Y$, as otherwise $G$ would contain Aut$(D,\bar{E})$. If $f$ behaves like $sw$ between them, then we simply replace $f$ by $sw_Y \circ f$. Note that one needs to check $sw_Y$ is a legitimate function. The reason it may not be is that applying $sw_Y$ could introduce a forbidden a tournament.  However, if it were the case that $sw_Y$ was illegitimate for this reason, then $f$ behaving like $sw$ between $X$ and $Y$ would also have been illegitimate for the same reason.

\textbf{Claim 3.} $f$ behaves like $id$ on every infinite orbit and between every pair of infinite orbits.

Suppose not, so there are infinite orbits $Y_1$ and $Y_2$ (possibly the same) and there are distinct $y_1, y_2 \in Y_1,Y_2$, respectively, such that $(y_1,y_2) \not\cong f(y_1,y_2)$. Now for any finite digraph $(a_1,a_2,\ldots,a_n) \in$ Forb$(\mt)$ with $(y_1,y_2) \cong (a_1,a_2)$, we can find $x_3,\ldots,x_n \in X$ such that $(y_1,y_2,x_3,\ldots,x_n) \cong (a_1,\ldots,a_n)$ (This statement can be verified analogously to \autoref{indorb}). Then $f$ has the effect of only changing what happens between $y_1$ and $y_2$, since we know $f$ behaves like $id$ on $X$ and between $X$ and all other infinite orbits. In short, given any finite digraph, we can use $f$ to change what happens between exactly two of the vertices of the digraph.

There are three options. If $f$ creates an edge from a non-edge, then we we can use $f$ to introduce a forbidden tournament, which gives a contradiction. If $f$ deletes the edge or changes the direction of the edge, then by \autoref{equalsSymD} or \autoref{understandingN}, as appropriate, we get that $G \geq$ Aut$(D,\bar{E})$.

\textbf{Claim 4.} $f$ behaves like $id$ between $\{c_1,c_2\}$ and the union of all infinite orbits.

The follows immediately from \autoref{constants}, composing with $sw_{c_i}$ if necessary.

\textbf{Conclusion.} We can assume that $f$ behaves everywhere like the identity, except on $(c_1,c_2)$, where it maps an edge to a non-edge. But then we get that $G=$ Sym$(D)$ by \autoref{equalsSymD}, completing the proof of (i).

(ii) The proof follows exactly the same series of claims as in part (i) but with minor adjustments to how one starts and concludes. We go through one case as an example, leaving the rest to the reader.  We will show that if Aut$(D,E) < G \leq$ Aut$(D,\bar{E})$, then $G> \langle - \rangle$ or $G > \langle sw \rangle$ (if they exist). Hence $G$ preserves non-edges but not the relation $E$. By \autoref{blackbox}, there is an edge $c_1c_2 \in D$ and a canonical function $f: (D,E,<,c_1,c_2) \to (D,E)$ which changes the direction of the edge $c_1c_2$. Suppose for contradiction that $G \not> \langle - \rangle$ and that $G \not> \langle sw \rangle$

Let $X$ be an independent orbit. By \autoref{oneorbit}, $f$ must behave like $id$ on $X$ and then by \autoref{xlessthany} and \autoref{xyinterdense}, $f$ must behave like $id$ between $X$ and all other infinite orbits. By repeating the argument of Claim 3 above, $f$ must behave like $id$ on the union of infinite orbits and so by \autoref{constants} $f$ must behave like $id$ between the constants and the union of infinite orbits. Now we are in the situation of \autoref{understandingN} (iv), so we conclude that $G \geq$ Aut$(D,\bar{E})$, so $G \geq \langle - \rangle, \langle sw \rangle$.

(iii) $(D,\bar{E})$ embeds every finite empty graph and is connected (\autoref{hensondigraphs} (ii)). Hence, if $(D,\bar{E})$ is a homogeneous graph then $(D,\bar{E})$ has to be the random graph or a Henson graph, by the classification of countable homogeneous graphs (\cite{lw80}).

Thus assume that $(D,\bar{E})$ is not a homogeneous graph. Let $G' \defeq$ max$\{$Aut$(D,E), \langle - \rangle, \langle sw \rangle$, $\langle -,sw \rangle\}$. Now let $G$ be a closed group  such that $G'<G \leq$ Sym$(D)$. We want to show that $G=$ Sym$(D)$. By \autoref{blackbox}, there are $\bar{c} \in D$ and a canonical $f:(D,E,<,\bar{c}) \to (D,E)$ such that $f$ cannot be imitated by any function of $G'$ on $\bar{c}$.  To be precise, we mean that for all $g \in G'$, $g(\bar{c}) \neq f(\bar{c})$.

Now we continue as in (i), proving that we may assume $f$ behaves like $id$ on the union of all infinite orbits and like $id$ between $\bar{c}$ and the union of infinite orbits.  In doing so, we may have composed $f$ with $-$ or $sw_A$ for some $A$. Since $-$ and $sw_A$ are elements of $G'$, these compositions do not change the fact that $f$ could not be imitated by $G'$ on $\bar{c}$. In particular, $f(\bar{c}) \not\cong \bar{c}$. Hence, we are in the situation of either \autoref{equalsSymD} (iv) or \autoref{understandingN} (iv). Thus, either $G=$ Sym$(D)$ and we are done, or $(D,\bar{E})$ is a homogeneous graph - contradiction.

Since Aut$(D,\bar E)$ contains $G'$ and is proper subgroup of Sym$(D)$, we get $G' = $Aut$(D,\bar E)$.
\end{proof}

\section{$2^{\aleph_0}$ pairwise non-isomorphic maximal-closed subgroups of Sym$(\mathbb{N})$}

\begin{Def} Let $G$ be a closed subgroup of Sym$(\mathbb{N})$. We say that $G$ is \emph{maximal-closed} if $G \neq$ Sym$(\mathbb{N})$ and there are no closed groups $G'$ such that $G < G' <$ Sym$(\mathbb{N})$.
\end{Def}

We construct $2^{\aleph_0}$ pairwise non-isomorphic maximal-closed subgroups of Sym$(\mathbb{N})$ by modifying Henson's construction of $2^{\aleph_0}$ pairwise non-isomorphic homogeneous countable digraphs and taking their automorphism groups.  The modification is needed to ensure that the groups are maximal. A short argument will show that the automorphism groups are pairwise non-conjugate. That these groups are pairwise non-isomorphic follows from Rubin's work on reconstruction.

In \cite{rub94}, Rubin showed that all Henson digraphs have a so called \emph{weak $\forall \exists$-interpretation}, which allows us to reconstruct the topology of the automorphism group. To be more precise, the automorphism groups of two Henson digraphs are isomorphic as abstract groups if and only if they are also isomorphic as topological groups. Since Henson digraphs have no algebraicity, we further know that their automorphism groups are topologically isomorphic if and only if they are conjugate.

Henson's construction in \cite{hen72} centres on finding an infinite anti-chain of finite tournaments. 

\begin{Def} Let $n \in \mathbb{N} \backslash \{0\}$. $I_n$ denotes the $n$-element tournament obtained from  the linear order $L_n$ by changing the direction of the edges $(i,i+1)$ for $i=1,\ldots,n-1$ and of the edge $(1,n)$.
\end{Def}

By counting 3-cycles, Henson showed that $\{I_n: n \geq 6\}$ is an anti-chain.  It is a short exercise to show that the 3-cycles in $I_n$  are $(1,3,n),(1,4,n),\ldots,(1,n-2,n),(1,2,3),(2,3,4),\ldots,(n-2,n-1,n)$.  In particular, we note that $I_n$ has at most two vertices through which there are more than five 3-cycles, namely the vertices 1 and $n$.

The automorphism groups of the Henson digraphs constructed by forbidding any subset of these $I_n$'s are not maximal: $\langle - \rangle$ and the automorphism group of the random graph are closed supergroups. By forbidding a few extra tournaments, however, we can ensure that the automorphism groups are maximal.

Let $T$ be a finite tournament that is not embeddable in $I_n$ for any $n$ and that contains a source but no sink. Such a $T$ can be found, for example, by ensuring there are at least three vertices through which there are more than five 3-cycles.  (A source, respectively sink, is a vertex which only has outgoing, respectively incoming, edges adjacent to it.)   

Let $k=|T|$. Then for $A \subseteq \mathbb{N}\backslash \{1,\ldots,k+1\}$, let $\mathcal{T}_A=\{I_n: n \in A\} \cup \{T': |T|=k+1, T$ is embeddable in $T'\}$. Then let $D_A$ be the Henson digraph whose set of forbidden tournaments is $\mt_A$. Observe that $\mt_A$ is an anti-chain. Hence, if $A,B \subseteq \mathbb{N}\backslash \{1,\ldots,k+1\}$ are not equal, $D_A \not\cong D_B$.

Now suppose for contradiction that Aut$(D_A)$ and Aut$(D_B)$ are conjugate. Let $f: D_A \to D_B$ be a bijection witnessing this, so that Aut$(D
_A)= f^{-1}$Aut$(D_B)f$. In particular this means that $f$ maps orbits of $D_A$ to orbits of $D_B$, i.e., that $f$ is canonical. Furthermore, $f$ is determined by its behaviour on 2-types. Now it is easy to see that $f$ cannot map edges to non-edges (true for all Henson digraphs) and that $f(p_E) \neq p_{E^*}$ (true by choice of $T$).  Thus, $f$ has the identity behaviour, so $f$ is an isomorphism, contradicting that $D_A \not\cong D_B$.

What remains is checking that Aut$(D_A)$ is maximal.

\begin{itemize}
\item $-$ nor $sw \in$ Sym$(D_A)$, since both these functions can turn sinks into sources.
\item $D_A$ embeds all finite linear orders, so $(D_A,\bar{E})$ is not $K_n$-free for any $n$, so $(D_A,\bar{E})$ is not a Henson graph.
\item Let $U \subset D_A$ be isomorphic to $T$ - this is possible as $T$ has not been forbidden. Then there is no vertex $x \in D$ such that for all $u \in U$, $E(x,u) \vee E(u,x)$, because all tournaments containing $T$ are forbidden. This implies $(D_A,\bar{E})$ is not isomorphic to the random graph.
\end{itemize}

Finally, by \autoref{main} we get:

\begin{thm} $\{$Aut$(D_A): A \subseteq \mathbb{N}\backslash \{1,\ldots,k+1\} \}$ is a set of $2^{\aleph_0}$ pairwise non-isomorphic maximal-closed subgroups of Sym$(\mathbb{N})$.
\end{thm}


\bibliographystyle{alpha}
\bibliography{bibliography}

\end{document}